\documentclass[11pt]{article}
\usepackage{amsmath, amssymb, latexsym, theorem}
\numberwithin{equation}{section}

\theoremstyle{plain}
\theorembodyfont{\itshape}
\newtheorem{theorem}{Theorem}[section]
\newtheorem{proposition}[theorem]{Proposition}
\newtheorem{lemma}[theorem]{Lemma}
\newtheorem{corollary}[theorem]{Corollary}

\theorembodyfont{\rmfamily}
\newtheorem{definition}[theorem]{Definition}

\newtheorem{remark}[theorem]{Remark}
\newenvironment{proof}{{\noindent \textbf{Proof}\,\,}}{\hspace*{\fill}$\Box$\medskip}

\title{On families of differential equations on two-torus with all phase-lock areas}
\author{Alexey Glutsyuk
\thanks{ CNRS, France (UMR 5669 (UMPA, ENS de Lyon) and UMI 2615 (Lab. J.-V.Poncelet)).
E-mail:
aglutsyu@ens-lyon.fr}
\thanks{National Research University Higher School of Economics (HSE), Moscow, Russia}
 \thanks{Supported by part by RFBR grants 13-01-00969-a, 16-01-00748, 16-01-00766
and  by ANR grant ANR-13-JS01-0010.}
Leonid Rybnikov\thanks{National Research University Higher School of Economics (HSE), International Laboratory of Representation Theory and Mathematical Physics, 20 Myasnitskaya st,
Moscow 101000, Russia,
and Institute for Information Transmission Problems, Moscow, Russia. \newline Email: leo.rybnikov@gmail.com}
\thanks{The work of L.R. was done at IITP under support of a grant from the Russian Science Foundation (project ? 14-50-00150).}}
\begin{document}
\maketitle

\begin{abstract} We consider two-parametric families of  non-autonomous ordinary differential equations on the two-torus with the coordinates $(x,t)$ of the type $\dot x=v(x)+A+Bf(t)$. We study its rotation number as a function of the parameters $(A,B)$.
The {\it phase-lock areas} are those level sets of the rotation number function $\rho=\rho(A,B)$ that have non-empty interiors.
V.M.Buchstaber, O.V.Karpov, S.I.Tertychnyi have studied the case, when $v(x)=\sin x$ in their joint paper. They
have observed the {\it quantization effect:}
for every smooth periodic function $f(t)$ the family of equations may have phase-lock areas only  for integer
rotation numbers. Another proof of this quantization statement
was later obtained in a joint paper by
Yu.S.Ilyashenko, D.A.Filimonov, D.A.Ryzhov. This implies the similar quantization effect for every
$v(x)=a\sin(mx)+b\cos(mx)+c$ and rotation numbers that are multiples of  $\frac 1m$.
We show that for every other analytic vector field $v(x)$ (i.e., having at least two
Fourier harmonics with non-zero non-opposite degrees and nonzero coefficients)  there exists an analytic periodic function $f(t)$ such that the corresponding family of equations has phase-lock areas for all the rational values of the rotation number.
\end{abstract}
\def\td{\mathbb T^2}
\def\ddx{\frac d{dx}}
\def\zz{\mathbb Z}
\def\rr{\mathbb R}
\def\var{\varepsilon}
\def\cc{\mathbb C}
\def\diff{Diff^{\infty}_+(S^1)}

\section{Introduction}

\subsection{Main result}

We consider families of ordinary differential equations on the torus $\td=\rr^2\slash2\pi\zz^2$ with the coordinates $(x,t)$ of the type
\begin{equation}\dot x=v(x) +A+Bf(t); \ A,B\in\rr \text{ are the parameters.}\label{2}\end{equation}
Family (\ref{2}) with
$v(x)=\nu \sin x$, $f(t)=\cos t$  arises in several models in physics, mechanics and geometry. For example, it describes
the overdamped model of the
Josephson junction (RSJ - model)  in superconductivity (our main motivation), see \cite{josephson, stewart, mcc, bar, schmidt};
it arises in  planimeters, see  \cite{Foote, foott}.

The flow map for the period $2\pi$ of equation (\ref{2})
 is a circle diffeomorphism depending on the parameters  $(A,B)$. We study its rotation number
$\rho=\rho(A,B)$ as a function of the parameters $A$ and $B$. (Normalization convention: the rotation number of
a rotation equals the angle of rotation divided by $2\pi$.) The {\it $r$-th phase-lock area} is the level set
$$\{(A,B) \ | \ \rho(A,B)=r\}\subset\rr^2$$
provided it has a non-empty interior.

 In 2001 J.Guckenheimer and Yu.S.Ilyashenko \cite{guck-il} studied the family
\begin{equation} \dot x=\var^{-1}(a-\cos x- \cos t) \label{jos}\end{equation}
as a slow-fast system with small parameter $\var>0$. They observed the following new effect:
 there exists a sequence of intervals $C_n\subset\{\var>0\}$, $C_n\to0$ as $n\to\infty$,
such that for every $\var\in C_n$ equation (\ref{jos}) has an attracting canard limit cycle. They have conjectured that this effect observed in
a very special family (\ref{jos}) should be observed in a generic slow-fast system on two-torus.

In the same year 2001 the family of equations (\ref{2}) with $v(x)=\sin x$, $f(t)=\sin t$ similar to  (\ref{jos})
was studied by V.M.Buchstaber, O.V.Karpov, S.I.Tertychnyi  \cite{buch2}  from a different point of view.
They have observed the following quantization effect of the rotation number. Another proof of this quantization statement
was later obtained in  \cite{IRF, LSh2009}.

\begin{theorem} \label{sin} \cite{buch2,IRF,LSh2009} Let
\begin{equation} v(x)=a\sin mx +b\cos mx+c, \ a,b,c\in\rr, \  m\in\zz.\label{v}\end{equation}
Then for every smooth  function $f:S^1\to\rr$ the corresponding family of equations (\ref{2}) may have phase-lock areas only for those values
of the rotation number that are integer multiples of $\frac 1m$.
\end{theorem}

\begin{remark} This theorem was stated and proved  in loc.cit. for $m=1$.  Its generalization to arbitrary integer values of $m$ then follows immediately by passing
to the quotient of the circle by the group of rotations by angles $\frac{2\pi k}m$, $k\in\zz$. The quotient projection transforms a field
$v(x)$ as in (\ref{v}) into the corresponding field with $m=1$.
\end{remark}

Our main result is the  theorem below showing that for any other analytic vector field $v$ the statement of Theorem \ref{sin}
  becomes drastically  false.

\begin{theorem} \label{a} Let $v$ be an analytic vector field on the circle $S^1=\rr\slash2\pi\zz$ different from (\ref{v}).
Then there exists an analytic function $f:S^1\to\rr$ such that the corresponding family of
differential equations (\ref{2}) has phase-lock areas for all the rational values of the rotation number.
\end{theorem}

\begin{remark} A vector field on $S^1$ is not of type (\ref{v}), if and only if
 its Fourier series has at least two  harmonics of nonzero non-opposite degrees with nonzero coefficients:
 \begin{equation}
 v(x)=\sum_ka_k e^{ikx}, \ a_k\in\cc;  \ a_l, a_n\neq0 \text{ for some } l,n\neq0, \ l\neq\pm n.\label{fourier}\end{equation}
 \end{remark}

\begin{remark} Each phase-lock area intersects each horizontal line $B=const$ by a segment (or a point). This follows from continuity of the rotation
number and its monotonicity in $A$ (see Proposition \ref{propmon} below). If the field $v$ is analytic, then for every given small $\var>0$
the intersections of the
phase-lock areas with the strip $|B|<\var$ also have non-empty interiors and are usually  called the {\it Arnold tongues}. The classical
Arnold tongue picture
for the Arnold family of circle diffeomorphisms can be found in \cite[p. 110]{arnold}. The statement  of
Theorem \ref{a} implies that the corresponding family of equations restricted to the latter strip has  Arnold tongues for all the rational values of the rotation number.
\end{remark}
\subsection{Plan of the paper and  sketch of proof of Theorem \ref{a}}

The definition and basic properties of the rotation number are recalled in the next subsection. Theorem \ref{a} is proved in Section 2.

Everywhere below for a vector field $w$ by $g^t_w$ we denote its time $t$ flow map.
By $(g^t_{w_1})_*w_2$ we denote the image of the field $w_2$ under the time $t$ flow map of
the field $w_1$.

Let $v$ be a vector field on $S^1$. Denote $m=m(v)$ the greatest common divisor of
the degrees of its Fourier harmonics with non-zero coefficients:
\begin{equation}m(v)=G.C.D.\{ k\in\zz\setminus 0; \ a_k\neq0\}\in\mathbb N, \text{ see (\ref{fourier}).}\label{mv}\end{equation}
We say that $v$ {\it has coprime harmonics,} if $m(v)=1$.

\begin{remark}
The group $\zz_m$ acts on $S^1$ by rotations of order $m$. For $m=m(v)$
the quotient projection $S^1\mapsto S^1_m=S^1\slash\zz_m$ transforms the field $v$ to that with coprime harmonics: it divides
the degrees of its harmonics by $m$.
\end{remark}

Let $v$ be a vector field that is not of type (\ref{v}).
Without loss of generality  we consider that $v$ has coprime harmonics: one can achieve this by passing to the above quotient.
Let  $\Lambda$ denote the Lie algebra generated by the constant vector field $\ddx$, the commutator
$v'=[\ddx,v]$ and its images under all the circle rotations.

Theorem \ref{a} and its proof are motivated by the following theorem proved in Subsection 2.2.

\begin{theorem} \label{gene}  For every
$C^{\infty}$-smooth
vector field $v$ on $S^1$  having coprime harmonics
there is the following alternative:

- either $v(x)=a\sin x+b\cos x +c$ and the Lie algebra $\Lambda$
coincides with that of M\"obius group $PSL_2(\rr)$ of conformal automorphisms of the unit disk;

- or $\Lambda$ is $C^{\infty}$- dense in the space $\chi(S^1)$ of $C^{\infty}$ vector fields on $S^1$.
\end{theorem}

Thus, in our case, when $v$ is not of type (\ref{v}),  $\Lambda$ is dense in $\chi(S^1)$.

\begin{remark} A theorem of Sophus Lie \cite{lie}
 says  that every Lie algebra of smooth {\it germs} at 0 of vector fields on $\rr$
is either at most three-dimensional, or infinite-dimensional (see a recent proof  in \cite[p. 132]{rs}).
Some  classes of  infinite-dimensional Lie algebras of  germs of vector fields in any dimension
were classified by Elie Cartan, see \cite{demaz} for references and related results on filtered Lie algebras. A related result
on actions of groups of circle diffeomorphisms on tuples of points in terms of the  group structure was obtained in \cite[theorem A, p. 125]{rs}.
Specialists believe that Theorem \ref{gene} is known, but we have not found it in literature.
\end{remark}

Let $Diff^{\infty}_+(S^1)$ denote the group of
{\it positive} (i.e., orientation-preserving)
circle diffeomorphisms.
Let $G=G(v)\subset Diff^{\infty}_+(S^1)$ denote the group generated by flows of the vector fields $\ddx$ and $v$. It consists of products
$$g(t_1,\dots,t_k,\tau_1,\dots,\tau_k)=g^{t_1}_{\ddx}\circ g^{\tau_1}_v\circ\dots\circ g^{t_k}_{\ddx}\circ g^{\tau_k}_v.$$
Set
$$G^s=\{ g(t_1,\dots,\tau_k) \in G \ | \ \sum_j\tau_j=2\pi s\};$$
\begin{equation}
 G^1_+=\{ g(t_1,\dots,\tau_k)\in G^1 \ | \ \tau_j>0 \text{ for all } j\}.\label{g1}\end{equation}

\begin{definition} A circle diffeomorphism $\phi:S^1\to S^1$ {\it belongs to the class F(v)}, if it can be realized as a period $2\pi$ flow map of
a differential  equation of the type (\ref{2}) with
$f(t)$ being a trigonometric polynomial.
\end{definition}

Recall that a $q$-periodic point $x_0$  of a
positive
circle diffeomorphism
$f$ is {\it hyperbolic,} if its multiplier $(f^q)'(x_0)$ is different from 1; then its periodic orbit is also called hyperbolic.
A
positive
circle diffeomorphism is {\it hyperbolic,} if it has periodic orbits and all of them are hyperbolic.

For any family of differential equations depending on parameters analytically, the condition that the corresponding flow map has a given rational rotation number and at least one hyperbolic periodic orbit is an open condition on the parameters. Thus, to prove Theorem \ref{a}, we have to show that for every $\rho\in\mathbb Q$ there exists at least some class F(v) diffeomorphism
with the rotation number $\rho$ and at least one hyperbolic orbit.

The proof of Theorem \ref{a} is done as follows.

Step 1.   We show (Proposition~\ref{approx} in Subsection 2.1)  that the class $F(v)$ diffeomorphisms  accumulate to all of $G^1_+$.

\def\mcg{\mathcal F}

Step 2. We show that the set $G^1\supset G^1_+$
  is $C^{\infty}$-dense in $Diff^{\infty}_+(S^1)$ (Corollary \ref{densecol}).
  This basically follows from density of the Lie algebra $\Lambda$ (Theorem \ref{gene}).  In more detail,
  the density of the algebra $\Lambda$ implies that the group $\mcg$ generated by the flows of its generators $\ddx$ and $[\ddx,v]$ is dense
  (Theorem \ref{densedif}).
  This implies that the group $G^0$ is dense, since
   its closure contains
   all of $\mcg$ (Proposition \ref{propd}).
  This immediately implies density of $G^1$.

Step 3. We show (Lemma \ref{grot} in Subsection 2.3) that
 for every rational number $\frac pq$ the set $G^1_+$ contains a diffeomorphism with rotation number $\frac pq$ and at least
one hyperbolic periodic orbit (either attractor, or repeller). We already know (step 2) that  $G^1$  is dense and hence,
contains hyperbolic circle diffeomorphisms with all the rotation numbers (since the existence of a hyperbolic orbit with a given rotation number is an open condition). We show that each one of them can be
analytically deformed in $G^1$ to an element of
the set $G^1_+$ without changing rotation number and keeping one and the same point periodic. The multiplier of the corresponding
periodic orbit will be an analytic function of the deformation parameters that is not identically equal to one. Thus, we can achieve that the
deformed element in $G^1_+$ have the given rotation number and a hyperbolic orbit. This will prove Lemma \ref{grot}.
 This is the main place where we use the analyticity of the field $v$.

Step 4. Constructing families (\ref{2}) with all phase-lock areas.
  Proposition~\ref{approx} and Lemma~\ref{grot} together imply that for every $\frac pq\in\mathbb Q$ there exists a trigonometric polynomial $f(t)$ for which the
corresponding family of equations (\ref{2}) has phase-lock area with rotation number $\frac pq$. We then deduce that for every $\frac pq$ there
exists an $N=N(\frac pq)$ such that for a generic\footnote{Everywhere in the paper by ``generic'' we mean ``topologically generic'': belonging to
an open dense subset.} trigonometric polynomial $f(t)$ of degree at most $N$ the corresponding family
of equations (\ref{2}) has phase-lock area corresponding to the rotation number value $\frac pq$. This together with persistence of phase-lock areas
under small perturbations shows that for
an appropriate
analytic function $f:S^1\to \rr$ the corresponding family (\ref{2}) has phase-lock
areas for all the rational  values of the rotation number. This will prove Theorem \ref{a}.

\subsection{Rotation numbers and phase-lock areas}

Let $\td=\rr^2\slash2\pi\zz^2$ be a torus with coordinates $(x, t)$. Consider the flow given by the nonautonomous differential equation
\begin{equation}\dot x=\frac{dx}{dt}=\phi(x,t) \label{eq}\end{equation}
with smooth right-hand side. The time $t$ flow mapping is a diffeomorphism $h_t:S^1\to S^1$ of the space circle. Consider the universal covering
$$\rr\to S^1=\rr\slash2\pi\zz$$
over the space circle. The flow mappings of equation (\ref{eq}) can be lifted to the universal covering and induce a smooth family of
diffeomorphisms
$$H_{r,t}:\rr\times\{ r\}\to\rr\times\{ r+t\}, \ H_{r,0}=Id.$$
Recall that for every $(x,r)\in\rr\times S^1$ there exists a limit
\begin{equation} \rho=\lim_{n\to+\infty}\frac1{2\pi n} H_{r,2\pi n}(x)\in\rr,\label{rotation}\end{equation}
which depends neither on $r$, nor on $x$ and is called the {\it rotation number of the flow of equation} (\ref{eq}) (e.g., see \cite[p. 104]{arnold}).


Now consider an arbitrary analytic family of equations


\begin{equation} \dot x=\psi(x,t,B)+A, \  (x,t)\in\td, \ A,B\in\rr.\label{equ1}\end{equation}

\begin{proposition} \label{propmon} \cite[pp. 104, 109]{arnold} The rotation number $\rho=\rho(A,B)$ of the flow of equation (\ref{equ1}) is a continuous function of
the parameters $(A,B)$ and a nondecreasing function of $A$. If, for some parameter value, the flow mapping $h_{2\pi}=h_{A,B,2\pi}:
S^1\times\{0\}\to S^1\times\{0\}$ of equation (\ref{equ1}) has a periodic orbit of  period $q$, and the cyclic order of the orbit on the circle is the
same as for an orbit of the rotation $x\mapsto x+\frac pq$, $p\in\zz$, then the rotation number is equal to $\frac pq(mod\ \zz)$.
\end{proposition}

\begin{remark}
The {\it rotation number of an orientation preserving diffeomorphism} $h:S^1\to S^1$, $S^1=\rr\slash2\pi$
is defined analogously. Let us consider its lifting $H:\rr\to\rr$ to the universal
cover, which is uniquely defined up to post-composition with translation by an integer multiple of $2\pi$. The limit
$\rho=\lim_{n\to+\infty}\frac1{2\pi n}H^n(x)$ exists. It is independent on $x$, well-defined modulo $\zz$
 and is called the rotation number $\rho=\rho(h)$, see the same chapter in loc.cit.  The rotation number of the time $2\pi$ flow mapping
 of a differential equation (\ref{eq}) coincides modulo $\zz$ with the rotation number of the equation.
\end{remark}

\section{Realization of phase-lock areas: Proof of Theorem \ref{a}}

\subsection{Approximations of products of flows of vector fields by flows of periodic equations}

\begin{proposition} \label{approx} For every smooth vector field $v$ on $S^1$ the corresponding subset $G^1_+\subset\diff$
from (\ref{g1}) is contained in the $C^{\infty}$-closure of class $F(v)$ diffeomorphisms. In more detail, consider an arbitrary
$g\in G^1_+$, let $\rho=\rho(g)$ denote its rotation number. Then for every $\alpha\equiv\rho(mod \zz)$, $k\in\mathbb N$,  $\var>0$
 there exist a trigonometric polynomial
$f(t)$
such that the corresponding equation (\ref{2}) with $A=0$, $B=1$
has rotation number $\alpha$
and its time $2\pi$ flow map is $\var$-close to $g$ in the $C^k$ topology.
\end{proposition}

\begin{proof} Fix a diffeomorphism
$$g=g(t_1,\dots,t_k,\tau_1,\dots,\tau_k)=g^{t_1}_{\ddx}\circ g^{\tau_1}_v\circ\dots\circ g^{t_k}_{\ddx}\circ g^{\tau_k}_v; \ \tau_j>0; \sum_j\tau_j=2\pi.$$
Let $\rho\in S^1=\rr\slash\zz$ denote its rotation number. Fix an arbitrary $\alpha\equiv\rho(mod\ \zz)$.  Let us construct a trigonometric
polynomial $f(t)$
such that the flow of the corresponding equation (\ref{2}) with $A=0$, $B=1$
 approximates $g$ and has rotation number $\alpha$. Without loss of generality we assume that $t_j\neq0$. For  every $s\in\zz$,
 $0\leq s\leq k$, set
$$T_s=2\pi-\sum_{j\leq s}\tau_j;  \ 2\pi=T_0>T_1>\dots>T_k=0.$$

Fix a small $\delta>0$. Let us split the segment $[0,2\pi]$ into $4k$ segments
$$I_s=[T_{s},T_{s-1}-\delta-\delta^2], \ J_s=[T_{s-1}-\delta-\delta^2,T_{s-1}-\delta];$$
$$ R_s=[T_{s-1}-\delta,T_{s-1}-\delta^2], \ V_s=[T_{s-1}-\delta^2, T_{s-1}]; \ s=1,\dots,k.$$
Set $I_{k+1}=I_1$, $J_{k+1}=J_1$, $R_{k+1}=R_1$, $V_{k+1}=V_1$. Note that the segments $I_j$ and $R_s$ are all disjoint, have size
of order at least $\delta$ (for small $\delta$) and are separated by segments $J_p$ and $V_i$ of size  $\delta^2$.
Let us choose a $C^{\infty}$ $2\pi$-periodic  function $\phi_{\delta}:\rr\to\rr$
satisfying the following conditions for every $s=1,\dots,k$:
\begin{equation} \phi_{\delta}|_{I_s}\equiv 0, \ \phi_{\delta}|_{R_s}\equiv\frac{t_s}{\delta}, \ \phi_{\delta}(J_s\cup V_s)\in[0,\frac{t_{s}}{\delta}].
\label{phi} \end{equation}
More precisely, we choose the function $\phi_{\delta}$ on the segments $J_s$ and $V_s$ as follows. Fix a $C^{\infty}$
function $\psi:[0,1]\to[0,1]$  such that $\psi|_{[0,\frac13]}\equiv0$, $\psi|_{[\frac23,1]}\equiv 1$. Set
$$\phi_{\delta}(t)=\frac{t_s}{\delta}\psi(\delta^{-2}(t-T_{s-1}+\delta+\delta^2)) \text{ for } t\in J_s;$$
$$\phi_{\delta}(t)=\frac{t_s}{\delta}\psi(\delta^{-2}(T_{s-1}-t)) \text{ for } t\in V_s.$$

\medskip

\begin{lemma}\label{delta limit} The $2\pi$ time flow mapping of the differential equation
\begin{equation}\dot x=v(x)+\phi_{\delta}(t)\label{equ}\end{equation}
tends to the mapping $g$ in $C^{\infty}$, as $\delta\to0$.
\end{lemma}

\begin{proof} The flow map of equation (\ref{equ}) through each segment $I_s$ coincides with $g^{\tau_s-\delta-\delta^2}_v$, which tends to
$g^{\tau_s}_v$, as $\delta\to0$. The flow through a segment $R_s$ tends to $g^{t_s}_{\ddx}$. Indeed, let us write down the
differential equation in the rescaled time variable $\hat t=\delta^{-1}(t-T_{s-1}+\delta)$. The rescaling transforms the segment
$R_s$ to $[0,1-\delta]$. The rescaled equation takes the form
$$\dot x=\delta v(x)+t_s.$$
Its flow through the latter segment obviously converges to $g^{t_s}_{\ddx}$. Let us show that the flow map of equation (\ref{equ}) through
each segment $J_s$ or $V_s$ tends to the identity. Let us prove this statement for a segment $V_s$: the proof for the segments $J_s$ is
the same. The time rescaling $\hat t=\delta^{-2}(t-T_{s-1}+\delta^2)$ transforms $V_s$ to the segment $[0,1]$ and equation (\ref{equ}) to
$$\dot x=\delta^2v(x)+\delta t_s\psi(1-\hat t).$$
Its time one flow mapping obviously converges to the identity together with derivatives, as $\delta\to0$. This proves
Lemma~\ref{delta limit}.
\end{proof}

The rotation number of equation (\ref{equ}) tends to $\beta\equiv\rho(g)(mod\zz)$, as $\delta\to0$, see Lemma \ref{delta limit}
and its proof.
This implies that
 replacing  $t_1$ in formula (\ref{phi}) by $t_1+2\pi n+\sigma$ with  $n=\alpha-\beta\in\zz$ and  an appropriate small
$\sigma=\sigma(\delta)$,
one can achieve that the rotation number of equation (\ref{equ}) be equal to the given $\alpha$. The corresponding flow map with
thus changed $\phi_{\delta}$ will again converge to $g$. Moreover, we can
 approximate each $\phi_{\delta}$ by a trigonometric polynomial $f_{\delta}(t)$
so that the flow map of the corresponding equation be close to that of equation (\ref{equ}) and the rotation number remains the same.
Finally we get a
family of analytic differential equations of type (\ref{2}) with  $f(t)$ being a trigonometric polynomial such that for $A=0$, $B=1$ its
rotation number equals $\alpha$ and  its time $2\pi$ flow mapping
is arbitrarily $C^{\infty}$-close to $g$. Proposition~\ref{approx} is proved.
\end{proof}

\subsection{Lie algebra generated by $v$ and $\ddx$ and group generated by their flows}
\def\ls{\overline\Lambda_s}
\def\H{\mathcal R}
Let $\H\simeq S^1$ denote the group of circle rotations.
\begin{theorem} \label{gen}
Let a $C^{\infty}$ vector field $v$ on $S^1$ be not of type (\ref{v}) and have coprime harmonics:  $m(v)=1$, see (\ref{mv}).
Then the Lie algebra $\Lambda$ generated by the vector field $\ddx$, the bracket
$v'=[\ddx,v]$ and  all the $\H$-images of the field $v'$ is dense in the space $\chi(S^1)$ of $C^{\infty}$ vector fields on $S^1$.
\end{theorem}

\begin{remark}
The condition that
$\Lambda$ contains the $\H$-images of the bracket $v'$  is needed only in the case, when $v$ is non-analytic: if $v$ is
analytic, then the latter images are automatically contained in the $C^{\infty}$-closure of the algebra $\Lambda$.
\end{remark}

\begin{proof} The  algebra $\Lambda$ is $\H$-invariant, by definition. The idea of the proof is to study the action of the group $\H$ on a closure of the algebra
$\Lambda$ (i.e., its adjoint action) and to split it as a sum of irreducible representations. For every $s\in\mathbb N$
let  $\ls$ denote the closure of the algebra $\Lambda$  in the Sobolev space $H_s(S^1)$ of vector fields.
The group $\H$ acts on each Sobolev space by unitary transformations, hence the Hilbert space $H_s(S^1)$ is naturally a unitary representation of the group $\H$. This representation is completely reducible, i.e. it splits into an orthogonal sum of irreducible representations, as does every unitary representation of the circle. The decomposition of $H_s(S^1)$ into the direct sum of irreducibles is $H_s(S^1)=\bigoplus\limits_{j\in\zz_{\ge0}} H_s(S^1)^j$, where $H_s(S^1)^0=\rr\ddx$ is one-dimensional, and $H_s(S^1)^j=\rr\cos jx\ddx+\rr\sin jx\ddx$ for $j>0$ are two-dimensional. Note that the irreducible representations $H_s(S^1)^j$ are pairwise non-isomorphic, and hence each $\H$-invariant subspace in $H_s(S^1)$ is the direct sum of some of the $H_s(S^1)^j$'s.

Since the Hilbert subspace $\ls\subset H_s(S^1)$ is an $\H$-invariant Hilbert subspace we have $\ls=\bigoplus\limits_{j\in\Gamma}H_s(S^1)^j$, where $\Gamma\subset\zz_{\ge0}$. Note that $$[H_s(S^1)^j,H_s(S^1)^k]=H_s(S^1)^{|j-k|}+ H_s(S^1)^{j+k}$$ for $j\ne k$ (this is an elementary calculation). Since the subspace $\ls\subset H_s(S^1)$
contains a dense subset $\Lambda$ that is
closed with respect to commutator operation, for any $j,k\in\Gamma$ such that $j\ne k$ we have $j+k\in \Gamma$ and $|j-k|\in \Gamma$.
Hence for the indexing set $\Gamma\subset\zz_{\ge0}$ we have two possibilities: either $\Gamma=\{0,m\}$ or $\Gamma=m\zz_{\ge0}$, where $m=G.C.D.(\Gamma)$.

Since $v$ is not of type (\ref{v}), $\Lambda$ is not contained in $H_s(S^1)^{0}+ H_s(S^1)^{m}$ for any $m\in \zz_{\ge0}$. Since $m(v)=1$, we have $m=G.C.D.(\Gamma)=1$. Therefore,
$\Gamma=\zz_{\ge0}$
 and $\ls=H_s(S^1)$ for every $s$. Thus, the Lie algebra $\Lambda$ is dense in the Sobolev space $H_s(S^1)$ for every $s$. This together with Sobolev's Embedding Theorem \cite[p. 411]{Ch} implies that $\Lambda$ is dense in  $\chi(S^1)$. This proves Theorem~\ref{gen}.
\end{proof}

\begin{theorem} \label{densedif} Let a family $\mathcal V$ of  $C^{\infty}$ vector fields on a compact manifold generate a Lie algebra
that is dense in the Lie algebra of all the $C^{\infty}$ vector fields. Then the group generated by the flows of the fields from $\mathcal V$
is dense in the space of all the $C^{\infty}$ diffeomorphisms isotopic to the identity.
\end{theorem}

The specialists believe that Theorem \ref{densedif} is known, but we have not found it in literature explicitly.

\begin{proof} Let $\Lambda$ denote the Lie algebra generated by $\mathcal V$.
The group generated by flows of the fields from the family $\mathcal V$ is obviously dense in the group generated
by flows of the fields from $\Lambda$. Hence, it is dense in the group $G$ generated by flows of all the smooth vector fields.
It suffices to show that  $G$ is dense.

The authors'
original proof of density of the group $G$
was to consider an isotopy of a given diffeomorphism $g$ to the identity and to show directly that
$g$ can be approximated by products of flows. To do this, we split the isotopy into small pieces of length $\var$: this expesses
$g$ as a composition of $N=O(\frac 1{\var})$ diffeomorphisms $h_j$ $\var$-close to the identity. We
approximate each $h_j$  by a flow; the approximation is of order $\var^2$. The products of approximating flows converge to $g$, as $\var\to0$.

Here we present another argument suggested by Yu.A.Neretin.
The  subgroup $G$ of the group of diffeomorphisms is normal. But the identity component of the group of diffeomorphisms
is simple (Thurston's theorem, see \cite[p.24, theorem 2.1.1]{ban}. Hence, the group $G$ is dense there. This proves
Theorem \ref{densedif}.
\end{proof}

Let $G^0\subset G$ denote the subset defined by (\ref{g1}). It is obviously a subgroup.

\begin{proposition}\label{propd} The closure in $\diff$ of the group $G^0$ contains the flows of both vector fields
$\ddx$ and $v'=[\ddx,v]$.
\end{proposition}

\begin{proof} The group $G^0$ itself contains the flow of the constant field $\ddx$, by definition.
For every $t\in\rr$ the composition $(g^{\frac1N}_{\ddx}\circ g^{\frac tN}_v\circ g^{-\frac1N}_{\ddx}\circ g^{-\frac tN}_v)^{N^2}$
belongs to $G^0$ and converges in $C^{\infty}$ to the time $t$ flow map of the
Lie bracket $v'=[\ddx,v]$,
as $N\to\infty$ (the classical argument). This implies the statement of  the proposition.
\end{proof}

\begin{corollary} \label{densecol} Let a smooth vector field $v$ on $S^1$ be not of type (\ref{v}) and have coprime harmonics: $m(v)=1$.
Then the sets $G^0$  and $G^1$ are dense in $Diff^{\infty}_+(S^1)$.
\end{corollary}

\begin{proof} Recall that by $\mcg\subset Diff^{\infty}_+(S^1)$ we denote the group generated by the flows of the fields $\ddx$ and $[\ddx,v]$.
The group $\mcg$ is dense in $\diff$. Indeed, it contains the
group generated by flows of the fields from the Lie algebra $\Lambda$, by definition and the same classical argument, as in the proof of
Proposition \ref{propd}.
This together with Theorems \ref{gen} and \ref{densedif} implies its density.  The $C^{\infty}$-closure of the group $G^0$
contains all of $\mcg$, by Proposition~\ref{propd}. Hence, $G^0$ is dense, as is $\mcg$. Thus
$G^1=G^0\circ g^1_v$
is also dense. The corollary is proved.
\end{proof}

\subsection{Construction of a family with all phase-lock areas}

Everywhere below, whenever the contrary is not specified,
 we consider that $v$ is an analytic vector field on $S^1$ that is not of type (\ref{v}) and has coprime collection of harmonics:
$m(v)\neq1$.

\begin{lemma} \label{grot} For every rational number $\frac pq$ the set $G^1_+$ contains a diffeomorphism with rotation number $\frac pq$
and at least one hyperbolic periodic orbit (either attractor, or repeller).
\end{lemma}

\begin{proof} Fix a hyperbolic diffeomorphism $h:S^1\to S^1$ with the rotation number $\frac pq$.
Fix an $l>2$ and a diffeomorphism $g_0=g(t_1^0,\dots,t_k^0,\tau_1^0,\dots,\tau_k^0)\in G^1$ that is $C^l$-close  to $h$  so that $g_0$ is also
a hyperbolic diffeomorphism with the same rotation number $\frac pq$. The latter $g_0$ exists by
Corollary \ref{densecol}.  Fix a $q$-periodic point $x_0\in S^1$ of the diffeomorphism $g_0$.
By definition, $g_0^q(x_0)=x_0$, $(g_0^q)'(x_0)\neq1$ (hyperbolicity). Let
$$T=T(t_2,\dots,t_k,\tau_1,\dots,\tau_k)$$
be the continuous function defined by the equations
\begin{equation}g(T,t_2,\dots,t_k,\tau_1,\dots,\tau_k)^q(x_0)=x_0, \
T(t_2^0,\dots,t_k^0,\tau_1^0,\dots,\tau_k^0)=t_1^0.\label{gtt}
\end{equation}

The function $T$ is analytic, and the corresponding diffeomorphisms\\ $g(T,t_2,\dots,t_k,\tau_1,\dots,\tau_k)$ have one and the same
rotation number $\frac pq$. This follows by definition, Proposition \ref{propmon} and implicit function theorem.
Recall that $\sum_j\tau_j^0=2\pi$. The multiplier $(g^q)'(x_0)$ is an analytic  function of the parameters. It is not identically equal to 1 on the
hyperplane $\sum_j\tau_j=2\pi$,
since it is different from one  for $g=g_0$ (since $g_0$ is hyperbolic). Therefore, for a generic collection of parameter values
$t_2,\dots,t_k,\tau_1,\dots,\tau_k$ in the latter hyperplane with $\tau_j>0$ the above multiplier is different from one. By construction, the
corresponding diffeomorphism $g$ from (\ref{gtt}) belongs to $G^1_+$, has rotation number $\frac pq$ and the orbit of the
point $x_0$ is $q$-periodic and hyperbolic. The lemma is proved.
\end{proof}

\begin{corollary} \label{corp} For every rational number $\frac pq$ there exists an $N=N(p,q)\in\mathbb N$ such that for a generic trigonometric
polynomial $f(t)$ of degree at most $N$ the family of equations (\ref{2}) has a phase-lock area corresponding to the
value $\frac pq$ of the rotation number.
\end{corollary}

\begin{proof}  Given a rational number $\frac pq$, let $g\in G^1_+$ be the corresponding diffeomorphism from Lemma \ref{grot}, which has
rotation number $\frac pq$ and a hyperbolic $q$-periodic orbit. For every $k\in\mathbb N$
there exists  a trigonometric polynomial $f(t)$  such that the time $2\pi$ flow mapping of the corresponding equation
(\ref{2}) with $A=0$, $B=1$ is arbitrarily $C^k$-close to $g$ and
{\it the equation} has the same rotation number $\frac pq$
(Proposition~\ref{approx}).
In particular, one can achieve that
the flow mapping has a periodic orbit with non-unit multiplier, as does $g$. This means that the parameters $(0,1)$ belong to the
interior of the level set $\rho=\frac pq$ of the rotation number, and hence, family (\ref{2}) has a phase-lock area corresponding to the value
$\frac pq$ of the rotation number. Let $N$ be the degree of the trigonometric polynomial $f(t)$. The latter statement implies that
for an open and dense set of trigonometric polynomials $f(t)$ of degree $N$ the corresponding family of equations (\ref{2})  has
$\frac pq$-th phase-lock area. Indeed, let $x_0$ be a hyperbolic $q$-periodic point of the
time $2\pi$ flow mapping
 corresponding to the above parameters
$(A_0,B_0)=(0,1)$.
Consider the continuous function $A=A(B,\phi)$ defined on the product of the set of real numbers $B$ and trigonometric polynomials
$\phi$ of degree $N$ by the
two following conditions: 1) $x_0$ is a $q$-periodic point of the time $2\pi$ flow map of equation (\ref{2}) with $f$
replaced by $\phi$; 2)
$A(1,f)=0$. The function $A(B,\phi)$ is well-defined and analytic, by monotonicity and implicit function theorem.
The rotation number of the corresponding equations (\ref{2}) with $A=A(B,\phi)$ is constant and equal to $\frac pq$, by construction and
continuity.  The multiplier $\mu=\mu(B,\phi)$ at $x_0$ of the
time $2\pi q$ flow mapping of the corresponding equation is also an analytic function of $(B,\phi)$;  $\mu\not\equiv 1$, since
$\mu(1,f)\neq 1$ by construction. Therefore for an open and dense set of trigonometric polynomials $\phi$ of degree $N$
the function $\mu(B,\phi)$ with fixed $\phi$ and variable $B$  is not identically equal to one. This means exactly that the family of
equations (\ref{2}) corresponding to $f=\phi$ has $\frac pq$-th phase-lock area. This proves Corollary \ref{corp}.
\end{proof}

\begin{proof} {\bf of Theorem \ref{a}.} First let us reduce the general case to the case, when $v$ has coprime collection of harmonics.
Let $v$ be an arbitrary analytic vector field on $S^1$ that is not of type (\ref{v}). In the case, when
the collection of its harmonics is not coprime, i.e., $m(v)=m\geq2$, the field $v$ is invariant under the action of the group $\zz_m$ of rotations
of order $m$. The quotient projection $S^1\to S^1_m=S^1\slash \zz_m\simeq S^1$ transforms $v$ to the quotient vector field $v_m$, which is
analytic, not of type (\ref{v}) and has a coprime collection of harmonics: passing to the quotient divides the Fourier harmonics by $m$.
In what follows we show that for
an appropriate analytic function  $f:S^1\to\rr$ the family of differential equations
\begin{equation}\dot x=v_m(x)+A+Bf(t)\label{xm}\end{equation}
on $S^1_m\times S^1$
has phase-lock areas for all the rational values of the rotation number.
A phase-lock area corresponding to
a rotation number $\rho$ in the family of equations (\ref{xm})  is a phase-lock area corresponding to the rotation number $\frac{\rho}m$
of the lifted family to $S^1\times S^1$.
 This  implies that the lifted family also has phase-lock areas for all the rational values of the rotation number.

Thus, from now on without loss of generality we consider that the field $v$ has a coprime collection of harmonics.

Let us numerate all the rational numbers by natural numbers: $\frac{p_1}{q_1}, \frac{p_2}{q_2},\dots$
There exists a trigonometric polynomial $f_1(t)$ of some degree $N_1$ for which the corresponding family of equations (\ref{2}) has a phase-lock area corresponding to $\rho=\frac{p_1}{q_1}$ (Corollary \ref{corp}). We can approximate it by a trigonometric polynomial $f_2$
of certain degree $N_2>N_1$ such that the corresponding family (\ref{2}) has $\frac{p_2}{q_2}$-th phase-lock area etc. On each step we choose
approximating trigonometric polynomial $f_{j+1}$  so close to $f_j$  that the previously
constructed phase-lock areas persist and the trigonometric polynomials $f_j$ thus constructed converge to an analytic function $f$.
The family (\ref{2}) corresponding to the function $f$ has phase-lock areas for all the rational values of the rotation number, by construction.
Theorem \ref{a} is proved.
\end{proof}

\section{Acknowledgements}

We are grateful to V.M.Buchstaber for the statement of problem. We are grateful to him and to
E.Ghys, Yu.A.Neretin, J.Rebelo for helpful discussions. We are grateful to the referee and the editors for helpful remarks.

\end{document}